\documentclass[11pt,reqno]{amsart}
\usepackage{amsfonts,amsmath,amssymb}
\usepackage[T2A]{fontenc}

\newtheorem{thm}{Theorem}[section]
\newtheorem{proposition}[thm]{Proposition}
\newtheorem{corollary}[thm]{Corollary}
\newtheorem{lemma}[thm]{Lemma}
\newtheorem{definition}[thm]{Definition}

\newtheorem{remark}[thm]{Remark}

\usepackage{color}
\usepackage{pstricks}
\usepackage{pst-node}
\usepackage{graphicx}
\usepackage{epstopdf}

\title[Supergroup $OSP(2,2n)$ and  super Jacobi polynomials] {Supergroup $OSP(2,2n)$ and  super Jacobi polynomials }

\author{G.S. Movsisyan}
\address{Department of Mathematics, Saratov State University, Astrakhanskaya 83, Saratov 410012, Russia}
\email{Movsisyangs@gmail.com}

\author{A.N. Sergeev}\address{Department of Mathematics, Saratov State University, Astrakhanskaya 83, Saratov 410012  and National Research University Higher School of Economics, Russian Federation.}
 \email{SergeevAN@info.sgu.ru, asergeev@hse.ru}

\begin{document}

\maketitle
\begin{abstract}  Coefficients of super Jacobi polynomials  of type $B(1,n)$ are rational functions  in three parameters $k,p,q$. At the point $(-1,0,0)$ these coefficient may have poles. Let us set  $q=0$ and  consider pair  $(k,p)$ as a point of  $\Bbb A^2$. If  we    apply blow up procedure  at the point $(-1,0)$ then we get a new family  of polynomials  depending on parameter $t\in \Bbb P$. If $t=\infty$  then  we get  supercharacters of Kac  modules  for Lie supergroup $OSP(2,2n)$ and   supercharacters of irreducible modules  can be obtained  for   nonnegative integer  $t$  depending on highest weight. Besides we obtained supercharcters of projective covers as specialisation of some nonsingular  modification of super Jacobi polynomials.
\end{abstract}
\tableofcontents

\section{Introduction} Let  $J_{\lambda} (x,k,p,q)$  be the family of Jacobi polynomials for root system $BC_N$ (see \cite{Mac1}). It  easily follows from the orthogonality relations that at the point  $(-1,0,0)$  Jacobi polynomials are well  defined and coincide with the corresponding characters of irreducible finite dimensional modules over symplectic Lie group $SP(2N)$. 

In this paper we investigate  the same  problem for super Jacobi polynomials. The main difficulty in this case is that  super Jacobi polynomials are not well defined at the point $(-1,0,0)$. One partial result on this problem was obtained in the paper \cite{SV2}. Namely 
 let $J_{\lambda}(x,y,k,p,q)$ be  the family  of super Jacobi polynomials in $m+n$ indeterminates  (see \cite{SV1}) labeled by the set of partitions $\lambda$ such that $\lambda_{m+1}\le n$.  It has been proved  in the paper \cite{SV2} that 
 $\lim_{(p,q)\rightarrow (0,0)}\lim_{k\rightarrow -1}J_{\lambda}(x,y,k,p,q)
 $  coincides up to sign with 
  Euler  supercharacters  (with a special choice of the parabolic subalgebras) of the Lie supergroup $OSP(2m,2n)$.  The goal of this paper is  to investigate further  in particular case $m=1$  possible relations  super Jacobi polynomials  with representation theory of the  Lie supergroup   $OSP(2,2n)$. The main result of the paper  can be formulated in the following way. Let us make a substitution $p=t(k+1)$ and take a limit  as $k\to-1$ of super Jacobi  polynomials. Then we get a new family of polynomials
  which depends rationally on $t$ and which  we will denote by $SJ_{\lambda}(t)$. Let $H(1,n)$ be the set of partitions such that $\lambda_2\le n$. We will call a diagram $\lambda$  {\it  singular} if   $\lambda_1-n=\lambda'_j+n-j$ for some  $1\le j\le n$ and otherwise  we will call it  {\it  regular}. The main result of the paper can be formulated in the following way:

 $1)$ if diagram $\lambda$ is regular  then $SJ_{\lambda}(t)$ does not depend on $t$ and coincides (up to sign) with the supercharacter of  irreducible module $L(\lambda)$ (see  Theorem \ref{Euler} and Remark \ref{irr});

$2)$ if diagram $\lambda$ is singular (and $j$ is the same as above)  then $SJ_{\lambda}(\lambda'_j)$ is well defined and coincides (up to sign) with the supercharacter of irreducible module $L(\lambda)$  (see Corollary \ref{irr1});

 In this paper   we use  two main properties of super Jacobi polynomials. The first one is that they are eigenfunctions of the deformed Calogero - Moser - Sutherland operator and the second property is that they satisfy the Pieri identity.  So instead of calculating the limit of the super Jacobi polynomials  we  calculate the limit of the CMS operator (which  is trivial) and the limit of the coefficients of the Pieri formulae. Our main tool is translation functors which were defined in the paper \cite{SV3}.

\section{Super Jacobi Polynomials}
In this section we define super Jacobi polynomials using the fact that they satisfy  the Pieri formula and that they are eigenfunctions of the deformed Calogero - Moser - Sutherland  operator.  We will always suppose in this paper that $q=0$. The deformed CMS operator  of  type   $BC_{m,n}$ has the following form (see \cite{SV1}, page 1712)
\begin{multline*}
\mathcal{L}=\sum_{i=1}^m\partial_{x_i}^2+k\sum_{j=1}^n\partial_{y_j}^2-k\sum_{i<j}^m\left(\frac{x_i+x_j}{x_i-x_j}(\partial_{x_i}-\partial_{x_j})+\frac{x_ix_j+1}{x_ix_j-1}(\partial_{x_i}+\partial_{x_j})\right)\\
-\sum_{i<j}^n\left(\frac{y_i+y_j}{y_i-y_j}(\partial_{y_i}-\partial_{y_j})+\frac{y_iy_j+1}{y_iy_j-1}(\partial_{y_i}+\partial_{y_j})\right)\\
-\sum_{i=1}^m\left(p\frac{x_i+1}{x_i-1}+2q\frac{x_i^2+1}{x_i^2-1}\right)\partial_{x_i}-k\sum_{j=1}^n\left(r\frac{y_j+1}{y_j-1}+2s\frac{y_j^2+1}{y_j^2-1}\right)\partial_{y_j}\\
-\sum_{i,j}\left(\frac{x_i+y_j}{x_i-y_j}(\partial_{x_i}-k\partial_{y_j})+\frac{x_iy_j+1}{x_iy_j-1}(\partial_{x_i}+k\partial_{y_j})\right),
\end{multline*}
where  $\partial_{x_i}=x_i\frac{\partial}{\partial{x_i}},\,\partial_{y_j}=y_i\frac{\partial}{\partial{ y_j}}$ and the parameters  $k,p,q,r,s$ satisfy the relations
$
p=kr,\quad 2q+1=k(2s+1)
.$
 In the formulae below we always suppose that   $h=-km-n-\frac12p-q$ where  $m,n$  are  non negative  integer numbers. In order to define coefficients of the Pieri formulae let us introduce the following notations:
let $H(m,n)$  be the set of partitions   $\lambda$ such that $\lambda_{m+1}\le n$
and 
\begin{gather*}
n(\lambda)=\lambda_2+2\lambda_3+\dots,\quad
c_{\lambda}=2n(\lambda')+2kn(\lambda)+|\lambda|(2h+2k+1),\\
a_i=\lambda_i+ki,i=1,2,\dots,\quad
c^{0}_{\lambda}(\square, x)=j-1+k(i-1)+x,\\
c^{-}_{\lambda}(\square,x)=\lambda_i-j-k(\lambda'_j-i)+x,\quad
c^+_{\lambda}(\square,x)=\lambda_i+j+k(\lambda'_j+i)+x,\\
\end{gather*}
where   $\lambda'$  denote  the conjugate partition  to $\lambda$ and $\Box$ denote the box $(ij)$.  
Let us set 
\begin{gather*}
C^{0}_{\lambda}(x)=\prod_{\Box\in\lambda}c^0_{\lambda}(\Box,x)
,\,\,\,\,C^{-}_{\lambda}(x)=\prod_{\Box\in\lambda}c^-_{\lambda}(\Box,x)
,\,\,\,\,C^{+}_{\lambda}(x)=\prod_{\Box\in\lambda}c^+_{\lambda}(\Box,x),\\
 J_{\lambda}(1)=4^{\lambda}\frac{C^0_{\lambda}(h+\frac12p+q)}{C^{-}_{\lambda}(-k)}\frac{C^0_{\lambda}(k+h-\frac12p+\frac12)}{C^{+}_{\lambda}(2h-1)}.
\end{gather*}

Let us also denote by   $S^+(\lambda)$  the set of partitions  $\mu$ which can be obtained from  $\lambda$ by adding one box and by  $S^-(\lambda)$ we will denote the set of partitions $\mu$ which can be obtained from  $\lambda$ by deleting one box.  Let us also set
$
S(\lambda)=S^+(\lambda)\cup S^-(\lambda)\cup \{\lambda\}
.$
If  $\mu\in S(\lambda)^+$   and  $\mu_i=\lambda_i+1$ then we set
$$
V_{\mu}(\lambda)=\prod_{j\ne i}^{l(\lambda)+1}\frac{(a_i-a_j-k)}{(a_i-a_j)}\frac{(a_i+a_j+2h-k)}{(a_i+a_j+2h)}
$$
$$
\times \frac{(a_i-k+h+\frac12p+q)}{(a_i-k(l(\lambda)+2))}\frac{(a_i+k(l(\lambda)+1)+2h)}{(a_i+h)}\frac{(a_i+h-\frac12p+\frac12)}{(a_i+h+\frac12)}.
$$
If   $\mu\in S^-(\lambda)$ and  $\mu_i=\lambda_i-1$ then we set 
$$
V_{\mu}(\lambda)=\prod_{j\ne i}^{l(\lambda)}\frac{(a_i-a_j+k)}{(a_i-a_j)}\frac{(a_i+a_j+k+2h)}{(a_i+a_j+2h)}
$$
$$
\times \frac{(a_i+k+h-\frac12p-q)}{(a_i+k(l(\lambda)+1)+2h)}\frac{(a_i-kl(\lambda))}{(a_i+h)}\frac{(a_i+h+\frac12p-\frac12)}{(a_i+h-\frac12)},
$$
and 
\begin{gather*}
a_{\lambda,\mu}=V_{\mu}(\lambda)\frac{ J_{\lambda}(1)}{ J_{\mu}(1)},\quad\mu\in S(\lambda)\setminus\{\lambda\},\quad
a_{\lambda,\lambda}=-k^{-1}(2h+p+2q)-\sum_{\mu\in S(\lambda)\setminus\{\lambda\}}V_{\mu}(\lambda).
\end{gather*}

We also need some more simple expression  for coefficient $a_{\lambda,\mu}$. We have $a_{\lambda,\mu}=a^{(1)}_{\lambda,\mu}a_{\lambda,\mu}^{(2)}$ where $a^{(1)}_{\lambda,\mu}$ does not depend on $p$ and $(i,j)$ is the added box and
\begin{gather}
a_{\lambda,\mu}^{(2)}=\frac14\prod_{r\ne i}^{l(\lambda)+1}\frac{(a_i+a_r+2h-k)}{(a_i+a_r+2h)}\notag\\
\times \frac{(a_i+k(l(\lambda)+1)+2h)}{(a_i+h)}\frac{(a_i+h-\frac12p+\frac12)}{(a_i+h+\frac12)}
\frac{2j-1+2ki+2h}{(j-1+ki+h-\frac12p+\frac12)}\label{add}\\
\prod_{s=1}^{j-1}\frac{a_i+k\lambda'_s+s+2h}{a_i+k\lambda'_s+s+2h-1}
 \prod_{r=1}^{i-1}\frac{a_r+j+k\lambda'_j+2h-1+k}{a_r+j+k\lambda'_j+2h-1}\notag
\end{gather}

Let $\mu\in S^-(\lambda)$ then we have  (where $(ij)$ is the deleted box)
\begin{gather}
a_{\lambda,\mu}^{(2)}=4\prod_{r\ne i}^{l(\lambda)}\frac{(a_i+a_r+2h+k)}{(a_i+a_r+2h)}\notag\\
\times \frac{(a_i+k+h-\frac12p-q)}{(a_i+k(l(\lambda)+1)+2h)}\frac{1}{(a_i+h)}\frac{(a_i+h+\frac12p-\frac12)}{(a_i+h+\frac12)}\notag\\
\times\frac{(j-1+k(i-1)+h+\frac12p+q)(j-1+ki+h-\frac12p+\frac12)}{2j-1+2ki+2h}\label{del}\\
\prod_{s=1}^{j-1}\frac{a_i+k\lambda'_s+s+2h-2}{a_i+k\lambda'_s+s+2h-1}
\prod_{r=1}^{i-1}\frac{a_r+j+k\lambda'_j+2h-k-1}{a_r+j+k\lambda'_j+2h-1}\notag
\end{gather}

Let   $P_{n,m}=\Bbb C[x^{\pm1}_1,\dots,x^{\pm1}_m,y^{\pm1}_1,\dots, y^{\pm1}_n]$ be the algebra of Laurent polynomials in   $m+n$ indeterminates. Now we are ready to define super Jacobi polynomials.

 \begin{thm}  Let    $ak+bh+c\ne0$  for any  $a,b,c\in \Bbb Z$. Then there exists a unique family of polynomials  $J_{\lambda}=J_{\lambda}(x,y,k,p,q)\,\in P_{n,m},\lambda\in H(m,n)$ such that:
\begin{gather}\label{pieri1}
  J_{\emptyset}=1,\quad
 \mathcal L J_{\lambda}=c_{\lambda}J_{\lambda},\quad
p_1J_{\lambda}=\sum_{\mu\in S(\lambda)}a_{\lambda,\mu}J_{\mu},\,\,
 \end{gather}
where  $p_1=x_1+x_1^{-1}+\dots+x_m+x_m^{-1}+k^{-1}(y_1+y_1^{-1}+\dots+y_n+y_n^{-1}).$
 \end{thm} 
\begin{proof}  Let $ak+bh+c\ne0$  for any  $a,b,c\in \Bbb Z$. Then it is not difficult to verify that  from the conditions  $c_{\nu}= c_{\mu}$  and   $\mu,\nu\in S(\lambda)$  it follows that  $\mu=\nu$. Therefore the operator 
$$
\mathcal{L}^{\lambda}_{\mu}=\prod_{\nu\in S(\lambda)\setminus\{\mu\}}\frac{\mathcal L-c_{\nu}}{c_{\mu}-c_{\nu}},
$$
is well defined. So if a family of polynomials  $\{J_{\lambda}\}$ satisfy the conditions of the Theorem then from the last formula in (\ref{pieri1}) it follows that  

\begin{equation}\label{rec}
\mathcal{L}^{\lambda}_{\mu}(p_1J_{\lambda})=a_{\lambda,\mu}J_{\mu},
\end{equation}
and uniqueness can be proved by induction on the number of boxes in  $\mu$.
Existence follows from  \cite{SV1} section 7.
\end{proof}

\section{Translation functors}
Introduce some linear transformations which we call translation functors. Let $V$ be the linear span of the super Jacobi polynomials. Then we have the decomposition 
$$
V=\bigoplus_{i\in \Bbb Z}V_i,\quad V_i=< J_{\lambda}\mid c_{\lambda}(-1,0,0)=i>
$$
in other words  $V_i$ is the eigenspace  of the  of the operator $\mathcal L$  corresponding to the eigenvalue $i$.
Let us denote by $P_i$  the projector onto  subspace $V_i$ with respect to the above decomposition and define the linear transformation

\begin{equation}\label{trans}
F_i(f)=P_i(p_1f),\quad f\in V
\end{equation}
\begin{proposition} Let $f\in V_j$   and suppose that  $f$  has no poles at $(-1,0,0)$. Then the same is true for $F_i(f)$ for any $i\in\Bbb Z$.
\end{proposition}
\begin{proof} We have 
\begin{equation}\label{dec}
p_1V_j\subset  W_i\oplus W_{i_1}\oplus W_{i_2}\oplus\dots\oplus W_{i_r}
\end{equation}
where $W_i, W_{i_1}, W_{i_2},\dots, W_{i_r}$ are finite dimensional subspaces in $V,V_{i_1},\dots,V_{i_r}$. Let 
$$
W_{i}=<J_{\lambda}\mid\lambda\in S>,\,\, W_{i_1}=<J_{\lambda}\mid\mu\in T> 
$$
and set
$$
f_1(t)=\prod_{\mu\in T}(t-c_{\mu}(k,p))
$$
Then operator $\mathcal D_1=f_{1}(\mathcal L)$ acts as zero in $W_{i_1}$ and as a diagonal operator in $W_i$ with diagonal elements $d_{\lambda}=f_1(c_{\lambda}),\lambda\in S$. 
Now having in mind Cayley-Hamilton theorem we can define 
$$
\mathcal C_1=(-1)^{N+1}\frac{1}{\sigma_N}\left(\mathcal D_1^{N}-\sigma_1\mathcal D_1^{N-1}+\dots+(-1)^{N-1}\sigma_{N-1}\mathcal D_1\right)
$$
where $\sigma_1,\dots,\sigma_N$ stand for the elementary symmetric polynomials in $d_{\lambda},\lambda\in S$.
 From our assumptions we see that 
$\sigma_N\ne 0$ when $k=-1,p=0$.
We see that   $\mathcal C_1(W_{i_1})=0$ and  by the Cayley-Hamilton theorem 
$\mathcal C_1$ acts as the identity in $W_{i}$. 

In the same way we can construct operators $\mathcal C_2,\dots\mathcal C_r$ and define
$$
\mathcal C=\mathcal C_1\mathcal C_2\dots\mathcal C_r.
$$
Let 
$$
p_1f=g+g_1+\dots+g_r
$$
be the decomposition according to (\ref{dec}).
Applying to both sides of this equality the operator $\mathcal C$ we get 
$$
\mathcal C(p_1f)=g = Pr_F(p_1f).
$$
But $\mathcal C$ is a differential operator with coefficients that have no poles at $k=-1,p=0,q=0$, so  both sides must be regular at this point.
\end{proof}
From now on we suppose that $m=1$. 
  The following formulae give eigenvalue of super Jacobi polynomial  $J_{\lambda}$.
 $$
 c_{\lambda}=\sum_{(i,j)\in \lambda}[2(j-1+k(i-1))+1-2n-p],\quad  \tilde c_{\lambda}=\sum_{(ij)\in\lambda}(2j-2i+1-2n)
 $$
 where $\tilde c_{\lambda}$ is the value of $c_{\lambda}$ for $k=-1,p=q=0$.
 For convenience  we will use below a  notation $F_i=F_{\lambda}$ if $\tilde c_{\lambda}=i$.  
 \begin{lemma}\label{cond} Let  $\mu,\nu\in S(\lambda)$ and $\mu\ne\nu$. Then   $\tilde c_{\mu}= \tilde c_{\nu}$ if and only if the  following conditions are fulfilled 
$$\mu=\lambda\cup\square,\,\nu=\lambda\setminus\tilde \Box, \quad \,\, j-i+\tilde j-\tilde i=2n-1
$$
where $\Box=(i,j),\quad \tilde\Box=(\tilde i,\tilde j)$.
\end{lemma} 
\begin{proof}  The conditions $\tilde c_{\mu}= \tilde c_{\nu}$ is equivalent to the following equation
$$
(n(\mu')-n(\nu'))-(n(\mu)-n(\nu))+(|\mu|-|\nu|)(\frac12-n)=0
$$
Therefore two cases are  possible:  $|\mu|-|\nu|=0$  or $|\mu|-|\nu|=2$.
Let us consider the first case. Then we have 
$$
\mu=\lambda\cup\square,\,\, \nu=\lambda\cup\tilde\square\,\,\text{or}\,\,\mu=\lambda\setminus\square,\,\, \nu=\lambda\setminus\tilde\square
$$
besides
$$
[n(\mu')-n(\nu')]-(n(\mu)-n(\nu))=0
$$
or in the equivalent form
$
j-\tilde j-(i-\tilde i)=0
$ or $j-i=\tilde j-\tilde i$. This means that  boxes $\Box,\tilde\Box$ can be added to $\lambda$ and both are located on the same diagonal. Therefore $\Box=\tilde\Box$. In the same way we can consider the case $\mu=\lambda\setminus\square,\,\, \nu=\lambda\setminus\tilde\square.$

Let us consider the second case $|\mu|-|\nu|=2$. So $\mu=\lambda\cup\square$ and $\nu=\lambda\setminus\tilde \Box$ and we also have 
$$
n(\mu)=n(\lambda)+i-1,\, n(\mu')=n(\lambda')+j-1,\,\Box=(ij)
$$
and
$$
n(\nu)=n(\lambda)-\tilde i+1,\, n(\nu')=n(\lambda')-\tilde j+1,\,\Box=(\tilde i \tilde j)
$$
So we have 
$
j-i+\tilde j-\tilde i+2-2n=1.$
\end{proof}

 We need some combinatorics related to translation functors. 
 \begin{definition} Let $\lambda,\mu\in H(1,n)$. Let us set 
 $$
 F_{\lambda}(\mu)=\{\nu\in S(\mu)\mid \tilde c_{\lambda}=\tilde c_{\nu}\}
 $$
 \end{definition}
 \begin{definition} A diagram $\lambda\in H(1,n)$ is called singular if there exists $1\le j\le n$ such that $\lambda_1-n=\lambda'_j+n-j$. Otherwise the diagram $\lambda$ is called regular.
 \end{definition}
 \begin{definition} Let $\lambda$ be a singular diagram and $\lambda_1-n=\lambda'_j+n-j$. Define the number 
$$
r(\lambda)=|\{r\mid  j\le r\le n,\,\lambda'_r=\lambda'_j\}|
$$
and denote by $\lambda^{\sharp}$ the diagram which can be obtained from $\lambda$ by deleting $r(\lambda)$ boxes from the first row and $r(\lambda)$ boxes   from the row of index  $\lambda'_j$.
\end{definition}

\begin{definition} Let $\lambda_1\ge n$. Let us define by induction the set $\pi_{\lambda}$: If $\lambda_1\le n$ then $\pi_{\lambda}=\{\lambda\}$; if $\lambda_1>n$ then $\pi_{\lambda}=F_{\lambda}(\pi_{\mu})$, where $\mu$ can is obtained from $\lambda$ by deleting the box from the first row.
\end{definition}

 \begin{thm} \label{comb}The following statements hold true
 
 $1)$  Let  $\lambda_1,\mu_1\le n$ and $\mu$ can be obtained from $\lambda$ by deleting one box then 
 $$
 F_{\lambda}(\mu)=\{\lambda\}
 $$
 
 $2)$ Let $\lambda$  be a regular diagram and  $\lambda_1> n$ and  $\mu$ can be obtained from $\lambda$ by deleting one box from the first row then 
 $$
F_{\lambda}(\mu)= \{\lambda\}
 $$
 
 $3)$ Let $\lambda$ is  singular $\lambda_1-n=\lambda'_j+n-j$ and  $\mu$  can be obtained from $\lambda$ by deleting one box from the first row then 
 $$
 F_{\lambda}(\mu)=\begin{cases}\{\lambda\},\,\,\text{if}\,\,\lambda'_{j+1}=\lambda'_j\\
\{\lambda,\,\nu\},\,\,\text{if}\,\,\lambda'_{j+1}<\lambda'_j
 \end{cases}
 $$
 where $\nu$ can be obtained from $\mu$ by  deleting   one box from column of index $j$.
 
 $4)$ Let  $\mu$ be a singular   diagram  and $\lambda$   be the diagram which can be obtained from $\mu$ by adding one box to the first row. Then we have 
$$
 F_{\lambda}(\mu^{\sharp})=\begin{cases}\emptyset,\,\,\text{if}\,\,\lambda\,\,\text{is regular}\\
\{\lambda^{\sharp}\}\,\, \text{if}\,\,\lambda\,\,\text{is singular}
\end{cases}
$$

$5)$  Let $\lambda_1>n$ then
$$
\pi_{\lambda}=\begin{cases}\{\lambda\}\,\,\text{if}\,\,\lambda\,\,\text{is regular}\\
\{\lambda,\lambda^{\sharp}\}\,\,\text{if}\,\,\lambda\,\,\text{is singular}\end{cases}
$$
 \end{thm}
 \begin{proof}
 Let us prove statement $1)$. Clearly $\lambda\in S(\mu)$. If $\nu\in S(\mu)\setminus\{\lambda\}$ and $\tilde c_{\lambda}=\tilde c_{\nu}$ then by Lemma \ref{cond} $j-i+\tilde j-\tilde i=2n-1$. By our assumptions $j,\,\tilde j\le n$  therefore  $j-i+\tilde j-\tilde i-2n+1<0$ and we got a contradiction. The case $\tilde c_{\lambda}=\tilde c_{\mu}$ implies that  $2(\lambda_1-1)+1-2n=0$. But this is impossible and we proved the first statement.
 
 Now let us prove  statement $2)$. Again as before $\lambda\in S(\mu)^+$ and if $\tilde c_{\lambda}=\tilde c_{\nu}$ then by Lemma \ref{cond} $\lambda_1-1+ j - i=2n-1$, where $\nu=\mu\setminus (ij)$. So $\lambda_1-n=\lambda'_j+n-j$. If $j>n$ then we have $\lambda_1-n=\lambda'_j+n-j\le1+n-j\le0.$ this is a contradiction with the condition $\lambda_1>n$. Therefore 
 $ j\le n$ but this contradicts regularity and we proved the second statement.
 
 Let us prove statement $3)$. As before if there exists $\nu\in S(\lambda)$ such that $\tilde c_{\lambda}=\tilde c_{\nu}$ then $\nu=\lambda\setminus (ij)$. But if $\lambda'_j=\lambda'_{j+1}$ then we cannot delete the box from column of index $j$. So we have $F_{\lambda}(\mu)=\{\lambda\}$. And if $\lambda'_j>\lambda'_{j+1}$ we can delete the box from column of index $j$. Therefore in this case we have $F_{\lambda}(\mu)=\{\lambda,\,\nu\}$ and we proved the third statement.
 
 Now let us prove  statement $4)$. Let $r=r(\mu)$ then $\mu^{\sharp}$ can be obtained from $\mu$ by deleting the sels
 $$
 (\mu_1,1), (\mu_1-1),\dots,\, (\mu_1-r+1,1),\quad (\mu'_j,j),\,(\mu'_{j},j+1),\dots,(\mu'_j,j+r-1)
 $$
 and it is easy to check that $\tilde c_{\mu}=\tilde c_{\mu^{\sharp}}$. Therefor we see that $\tilde c_{\mu^{\sharp}}=\tilde c_{\mu}\ne \tilde c_{\lambda}$. So $\mu^{\sharp}\notin F_{\lambda}(\mu^{\sharp})$. Now let $\nu\in S^+(\mu^{\sharp})$ then 
 $$
 \tilde c_{\lambda}= \tilde c_{\mu}+2\mu_1+1-2n=\tilde c_{\nu}=\tilde c_{\mu^{\sharp}}+2(\tilde j-\tilde i)+1-2n
  $$
  Therefore $\mu_1=\tilde j-\tilde i$. Since $\mu_1>n$ then $\tilde j>n$ and therefore $\tilde i=1$ and $\tilde j=\mu_1^{\sharp}+1$ So we come to equality  $\mu_1=\mu_1^{\sharp}$ which is impossible since $\mu$ is singular. Now consider case when $\nu\in S^-(\mu^{\sharp})$. In the same way as before we come to equality 
$$
\mu_1+1-2n+\tilde j-\tilde i=0\,\,\text{or}\,\,\mu'_j-j-((\mu^{\sharp})'_{\tilde j}-\tilde j)=-1
$$
  Therefore $\tilde j<j$ and $(\mu^{\sharp})'_{\tilde j}=\mu'_{\tilde j}$. So $\tilde j=j-1$ and $\mu'_{j-1}=\mu'_j$. But the last equality contradicts to the regularity of $\lambda$. So we have $F_{\lambda}(\mu^{\sharp})=\emptyset$.
  If $\lambda$ is not regular then in the same way we get $F_{\lambda}(\mu^{\sharp})=\{\mu^{\sharp}\setminus(\mu'_j,j-1)$. And it is easy to check that  $ \mu^{\sharp}\setminus(\mu'_j,j-1)=\lambda^{\sharp}$ and we proved the forth statement.
  Now let us prove the statement $5)$ induction on $|\lambda|$.  Let $\lambda$ is regular an $\mu$ can be obtained from $\lambda$ by deleting box from the first row. If $\mu$ is regular the by induction $\pi_{\mu}=\{\mu\}$. Therefore by statement $2)$ we have $F_{\lambda}(\pi_{\mu})=\{\lambda\}=\pi_{\lambda}$. If $\mu$ is singular the by induction $\pi_{\mu}=\{\mu.\mu^{\sharp}\}$. Then by statement $2)$ we again have $F_{\lambda}(\mu)=\lambda$ and by statement $4)$ we have $F_{\lambda}(\mu^{\sharp})=\emptyset$. Therefore $F_{\lambda}(\pi_{\mu})=\pi_{\lambda}=\{\lambda\}$. At last consider the case when $\lambda$ is singular. If $\mu$ is regular than  by statement $3)$ we have $\pi_{\lambda}=\{\lambda,\nu\}$ and as it is easy to see  $\nu=\lambda^{\sharp}$. If $\mu$ is singular then by statement $4)$ we have $\pi_{\lambda}=\{\lambda,\lambda^{\sharp}\}$. Theorem is proved.
 \end{proof}
\begin{definition} Let us denote by $X_{\lambda}$ for singular diagram $\lambda$ the set 
$$
X_{\lambda}=\{\lambda,\lambda^{\sharp},(\lambda^{\sharp})^{\sharp},\dots \}
$$
\end{definition} 
 \begin{lemma} Let $\lambda$ be a singular diagram such that $\lambda_1-n=\lambda'_j+n-j$. Then set $X_{\lambda}$ consist of $\lambda'_j+1$  elements.
 \end{lemma}
\begin{proof} Let us induct on $\lambda'_j$. If $\lambda'_j=1$  then $\lambda'_{j+1}=\dots=\lambda'_n=1$ and $r(\lambda)=n-j+1$ and $\lambda_1=j-1$. Therefore $\lambda^{\sharp}$ is not singular  and $X_{\lambda}=\{\lambda, \lambda^{\sharp}\}$. So we check the first step of induction. Suppose that $\lambda'_j>1$. Then by definition $(\lambda^{\sharp})'_j=\lambda'_j-1$ and Lemma follows from inductive assumption.
\end{proof}

 \section{Nonsingular basis}
 
\begin{definition}\label{projective1} Let  $\lambda\in H(1,n)$ then we define by induction the family of polynomials $I_{\lambda}$ 
\begin{equation}
I_{\lambda}=\begin{cases} J_{\lambda},\,\,\text{if}\,\, \lambda_1\le n\\
F_{\lambda}(I_{\mu}),\,\,\text{if}\,\, \lambda_1>n
\end{cases}
\end{equation}
where $\mu$  is the diagram which can be obtained from $\lambda$ by deleting the last box from the first row. 
\end{definition}

\begin{thm}\label{well} Polynomials $ I_{\lambda}$ have no poles at the point  $k=-1, p=q=0$.
\end{thm} 
 
 \begin{proof} Let us prove the Theorem in the case when  $\lambda_1\le n$. Then by definition we have $I_{\lambda}=J_{\lambda}$ and we need to prove that these polynomials are well defined. We will use induction on $|\lambda|$.
 If $|\lambda|=0$ then $J_{\lambda}=1$ and the Theorem is obviously true. Let $|\lambda|>0$ and $\mu$ be the diagram obtaining from $\lambda$ by deleting the last box from the last row.  Then from Pieri formula (\ref{pieri1}) and by statement $1)$ of  Theorem \ref{comb} we have 
 $$
 F_{\lambda}(J_{\mu})=a_{\mu,\lambda}J_{\lambda}
 $$
 and we only need to prove that $a_{\mu,\lambda}$ has no poles  or zeroes at the point $(-1,0,0)$.  Let us check that.  We have explicit   expression (\ref{add})  for $a_{\mu,\lambda}$  (we need to permute $\mu$ and $\lambda$ in that formula). Case $l(\lambda)=1$ can be easily checked.  So suppose that $l(\lambda)=l>1$. We need to verify that  all factors in the nominator and denominator of $a^{(2)}_{\mu,\lambda}$  are non zero. Consider for example the product 
 $$
 \prod_{r\ne l}^{l(\lambda)+1}\frac{(a_l+a_r+2h-k)}{(a_l+a_r+2h)}.
 $$
We have $(a_l+a_1+2h-k)(-1,0,0)=\mu_l+\mu_1-2n+2-l<0$. Therefore for $r>1$ we have 
$$
(a_l+a_r+2h-k)(-1,0,0)<(a_l+a_1+2h-k)(-1,0,0)<0
$$
and for any $r$
$$
(a_l+a_r+2h)(-1,0,0)<(a_l+a_r+2h-k)(-1,0,0)<0
$$
All other factors can be checked in the same manner. 
 So we checked that $a_{\mu,\lambda}$ is well defined at the point $(-1,0,0)$ and is not zero.
 
 And by induction we proved the case when $\lambda\le n$.  The second statement  follows  from the previous one and the fact that translation functors  map regular polynomial to regular polynomial. Theorem is proved.
  \end{proof}

Now we are going to calculate explicitly the polynomials $I_{\lambda}$ in case $\lambda_1>n$.

\begin{lemma}\label{project}  Let $\lambda_1> n$. Then the following formulae hold true
 $$
 I_{\lambda}=\begin{cases}
  J_{\lambda},\,\,\text{if}\,\,\lambda\,\,\text{is a regular diagram}\\
 J_{\lambda}+b_{\lambda}J_{\lambda^{\sharp}},\,\,\text{if}\,\,\lambda\,\,\text {is  a singular diagram}
 \end{cases}
 $$
 where 
 $$
 b_{\lambda}=a_{\lambda^{(0)}\lambda^{(1)}} a_{\lambda^{(1)}\lambda^{(2)}}\dots a_{\lambda^{(r-1)}\lambda^{(r)}}
 $$
 where $r=r(\lambda)$ and  $\lambda^{(0)}$ is the diagram which can be obtained from $\lambda$ by deleting $r$ boxes from the first row  and $\lambda^{(s)}$ can be obtained from $\lambda^{(s-1)}$ by deleting one box from the row of index $\lambda'_j$ for $s=1,2,\dots, r$.
\end{lemma} 
\begin{proof}  First it is not difficult to verify that $a_{\lambda,\mu}=1$ if $\mu$ can be obtained from $\lambda$ by addind box to the first row. Then  formula (\ref{Pro}) follows  from  Theorem \ref{well}  statement $5)$ with some coefficient $b_{\lambda}$. If $\mu$ is the diagram which can be obtained from $\lambda$ by deleting one box from the first row then by Theorem \ref{comb}  statement $4)$ we have $F_{\lambda}(\mu^{\sharp})=\{\lambda^{\sharp}\}$. Therefore $b_{\lambda}=b_{\mu}a_{\mu^{\sharp},\lambda^{\sharp}}$ and we can apply inductive assumption.
\end{proof}
\begin{corollary}\label{formula}  Let $\lambda$ be a singular diagram and let us define by induction $\lambda^{s\sharp}=(\lambda^{(s-1)\sharp})^{\sharp}$. Then the following equality holds true
\begin{equation}\label{Pro}
J_{\lambda}=I_{\lambda}-b_{\lambda}I_{\lambda^{\sharp}}+b_{\lambda}b_{\lambda^{\sharp}}I_{\lambda^{2\sharp}}+\dots+(-1)^lb_{\lambda}b_{\lambda^{\sharp}}\dots b_{\lambda^{l\sharp}}I_{\lambda^{l\sharp}},
\end{equation}
where $l=\lambda'_j$.
\end{corollary}  
\begin{proof} We have 
$$
I_{\lambda}=J_{\lambda}+b_{\lambda}J_{\lambda^{\sharp}},\,\, I_{\lambda^{\sharp}}=J_{\lambda^{\sharp}}+b_{\lambda^{2\sharp}}J_{\lambda^{2\sharp}},\,\,\dots,\, I_{\lambda^{l\sharp}}=J_{\lambda^{l\sharp}}
$$
and Corollary follows.
\end{proof}

\section{Specialisation}

In order to   define  and calculate explicitly  polynomials   $SJ_{\lambda}(t)$  we need some additional preliminary results about rational functions.
 Let $\varphi(k,p)$ be a rational function of the form
$$
\varphi(k,p)=\frac{\prod_{i\in I}(p-\alpha_i)}{\prod_{j\in J}(p-\beta_j)}
$$
where $\alpha_i,\beta_j$ are linear functions in $k$.  We are going to calculate the following rational function 
$$
\varphi(t)=\lim_{k\to -1}\varphi(k,t(k+1))
$$

Let us represent $I=I_0\cup I_1$ where for $i\in I_0$ we have $\alpha_i=d_i(k+1)$ and for $i\in I_1$  $\alpha_i$ is not divisible on $k+1$. In the same way let us represent $J=J_0\cup J_1$ where for $j\in J_0$ we have $\beta_j=e_j(k+1)$ and for $j\in J_1$  $\beta_j$ is not divisible on $k+1$. So we can represent $\varphi(k,p)=\varphi_0(k,p)\varphi_1(k,p)$ where
$$
\varphi_0(k,p)=\frac{\prod_{i\in I_0}(p-\alpha_i)}{\prod_{j\in J_0}(p-\beta_j)},\quad \varphi_1(k,p)=\frac{\prod_{i\in I_0}(p-\alpha_i)}{\prod_{j\in J_0}(p-\beta_j)}.
$$ 
 
By definition the function  $\varphi_1(k,p)$ is well defined at the point $(-1,0)$.
 
 \begin{lemma}\label{special} `The following equalities hold true:
 
$1)$ if $\mid I_0\mid <\mid J_0\mid$ then $\varphi(t)=\infty$

$2)$ if $\mid I_0\mid >\mid J_0\mid$ then $\varphi(t)=0$

$3)$ if $\mid I_0\mid =\mid J_0\mid$  then
$$
\varphi(t)=\varphi_1(-1,0)\frac{\prod_{i \in I_0}(t-d_i)}{\prod_{j \in I_0}(t-e_j)}
 $$
\end{lemma}
\begin{proof} A proof easily follows from the definitions.
\end{proof}
\begin{remark} Suppose that function  $\varphi(-1,p)=f(p)$ is well defined and we know it
$
f(p)=p^c g(p)
$
 where $g(p)$ does not have zeros or poles at $p=0$. Therefore we see that $c=\mid I_0\mid -\mid J_0\mid$, and if $c=0$,   then $\varphi_1(-1,0)=g(-1)$ and 
$$
\varphi(t)=g(-1)\frac{\prod_{i \in I_0}(t-d_i)}{\prod_{j \in I_0}(t-e_j)}
$$
\end{remark}
\begin{remark} It is easy to check that
$$
\lim_{t\to\infty}\varphi(t)=\lim_{(p,q)\to (-1,0)}\lim_{k\to -1}\varphi(k,p,q)
$$
\end{remark}

For a function  $F(k)$  let us define  $\tilde F=\lim_{k\rightarrow-1}F(k)$. The following Theorem was proved in \cite{SZ}.

\begin{thm}\label{lim1} The following statements hold true

$1)$ If $\mu\in S^+(\lambda)$  then $\tilde a_{\lambda,\mu}=1$,

$2)$ 
 if  $\mu\in S^-(\lambda)$ and  $\mu_i=\lambda_i-1$ then 
$$
 \tilde a_{\lambda,\mu}=\frac{(2\tilde a_i+2\tilde h+p+2q)(2\tilde a_i+2\tilde h-p-2q-2)}{(2\tilde a_i+2\tilde h-1)^2}
 $$
 $$
 \frac{(2\tilde a_i+2\tilde h+p-1)(2\tilde a_i+2\tilde h-p-1)}{(2\tilde a_i+2\tilde h)(2\tilde a_i+2\tilde h-2)},
 $$

$3)$
$$
\tilde a_{\lambda,\lambda}=\frac{p(p+2q+1)}{2\tilde h-2l(\lambda)-1}+p-\sum_{i=1}^{l(\lambda)}\frac{2p(p+2q+1)}{(2\tilde a_i+2\tilde h-1)(2\tilde a_i+2\tilde h+1)}.
$$
\end{thm}

\begin{lemma} Let $\lambda,\mu$ be such diagrams that  $\mu=\lambda\setminus (i,j),\,\,1\le j\le n$.  

$1)$ If $\lambda_1>n,$ and $\mu_1-n=\mu'_r+n-r,\,\lambda'_r>1$ for some $1\le r\le n$  then
$$
a_{\lambda,\mu}(t)=\begin{cases}\displaystyle\frac{t-\lambda'_j+2}{t-\lambda'_j+1},\, \text{if}\,\,r=j\\
\quad1\,\,\quad\quad\quad\text{if}\,\,\,\,r>j
\end{cases}
$$

$2)$ If $\lambda_1\le n$ and  $i=1$   then
$$
a_{\lambda,\mu}=\begin{cases}\displaystyle\frac{2}{t},\,\,\text{if}\,\, j=n\\
1\,\,\text{if}\,\,j<n
\end{cases}
$$
\end{lemma}
\begin{proof}  Let us prove statement $1)$.  By Lemma \ref{special}  we need to calculate $(a_{\lambda,\mu})(-1,0,0)$ and  numbers $d,e$.   But   $(a_{\lambda,\mu})(-1,0,0)=1$  by Theorem 2  from  \cite{SZ}. Let us  consider all factors in $a_{\lambda,\mu}$ which depend on $p$. Let $(i,j)$ be the box such that $\mu=\lambda\setminus(i,j)$. Then we have $i=\lambda'_j,\, j=\lambda_i$. Since  $i>1$ then  it is not difficult to verify that  $a_{\lambda,\mu}$ has  one pole at  
 $a_i+a_1+2h$ and one zero at  $a_1+j+k\lambda'_j+2h-k-1$. Therefore we have 
 $$
 a_{\lambda,\mu}(t)=\frac{a_1+j+k\lambda'_j+2h(t)-k-1}{a_i+a_1+2h(t)}=\frac{(\mu'_j-1-t)(k+1)}{(\mu'_j-t)(k+1)}
 $$
If $r>j$ then at the point $(-1,0,0)$ we have 
$$
a_i+a_1+2h= \lambda_i-i+\lambda_1+1-2n=j-\lambda'_j+\mu'_r-r+1=j-r+\mu'_j-\mu'_r<0 
$$
Therefore  at the same point we have 
$$
\prod_{s\ne i}^{l(\lambda)}\frac{(a_i+a_s+2h+k)}{(a_i+a_s+2h)}\ne0
$$
Further  at the point $(-1,0,0)$ we have 
$$
a_1+j+k\lambda'_j+2h-k-1=\lambda_1-\lambda'_j+j+1-2n=\mu'_r-\mu'_j+j-r<0
$$
Therefore  at the same point we have 
$$
\prod_{s=1}^{i-1}\frac{a_s+j+k\lambda'_j+2h-k-1}{a_s+j+k\lambda'_j+2h-1}\ne0
$$
Also at the point $(-1,0,0)$ we have 
$$
a_i+j+k\lambda'_j+2h-1=2(\lambda_i-i-n)+1\le -3
$$
Therefore  at the same point we have 
$$
\prod_{s=1}^{i-1}\frac{a_s+j+k\lambda'_j+2h-2}{a_s+j+k\lambda'_j+2h-1}\ne0
$$
And it is also easy to check that all other factors strictly less then zero. Therefore  we proved the first statement.
The second  statement  can be proved in the same manner. 
\end{proof}

\begin{lemma}\label{koef} Let $\lambda$ be a  singular diagram  such that $\lambda_1-n=\lambda'_j+n-j$ then 
\begin{equation}
b_{\lambda}(t)=\begin{cases} \displaystyle\frac{2}{t}\,\,\text{if}\,\,\lambda'_j=1,\\
\displaystyle\frac{t-\lambda_j'+2}{t-\lambda'_j+1},\,\,\text{if}\,\,\lambda'_j>1
\end{cases}
\end{equation}
\end{lemma}
\begin{proof} It easily follows from the previous Lemma.
\end{proof}
\begin{definition} For any $t\in\Bbb C$  and $\lambda\in H(1,n)$ let us define
$$
SJ_{\lambda}(t)=\lim_{k\to -1}J_{\lambda}(k,t(k+1),0),\,\,\, SI_{\lambda}(t)=\lim_{k\to -1}I_{\lambda}(k,t(k+1),0)
$$

\end{definition}

\begin{corollary}\label{formulat} Let $\lambda\in H(1,n)$. Then

$1)$  $SI_{\lambda}(t)$ does not depend on $t$

$2)$ If $\lambda$ is a regular diagram then $SJ_{\lambda}(t)$ does not depend on $t$.

$2)$ If  $\lambda$ is a singular diagram then for  $t\notin \Bbb Z_{\ge0}$ the polynomial $SJ_{\lambda}(t)$ is well defined and  we have the following equality
\begin{equation}
SJ_{\lambda}(t)= SI_{\lambda}-\frac{t-l+2}{t-l+1} SI_{\lambda^{\sharp}}+\frac{t-l+3}{t-l+1} SI_{\lambda^{2\sharp}}+\dots+(-1)^{l}\frac{2}{t-l+1} SI_{\lambda^{l\sharp}}
\end{equation}
where $l=\lambda'_j$.
\end{corollary}
\begin{proof} Let us prove the first statement. Polynomial $I_{\lambda}$ is well defined at the point $(-1,0,0)$. Therefore  in the notations of the Lemma \ref{special} for any its coefficient $\varphi(k,p,q)$ we have $J_0=\emptyset$. Therefore we always have $|I_0|=|J_0|$ or $|I_0|>|J_0|$. So by Lemma \ref{special} we have  $\varphi(t)=const$ in the first case and $\varphi(t)=0$ in the second case and we proved the first statement.

By Theorem \ref{well} and  from the Definition \ref{projective1} it follows that for regular $\lambda$ we have $J_{\lambda}=I_{\lambda}$ and we get the second statement.

The third statement follows  from the  Corollary \ref{formula} and  Lemma  \ref{koef}. 

\end{proof}
Let us denote by $SJ_{\lambda}(\infty)$ the limit $SJ_{\lambda}(t)$ when $t\to\infty$ and from the Corollary \ref{formulat} we see that
\begin{equation}
SJ_{\lambda}(\infty)=SI_{\lambda}-SI_{\lambda^{\sharp}}+ SI_{\lambda^{2\sharp}}+\dots+(-1)^{l-1} SI_{\lambda^{(l-1)\sharp}}
\end{equation}

And  for  singular diagram we have 
\begin{equation}\label{projective}
SI_{\lambda}=\begin{cases} SJ_{\lambda}(\infty),\,\,\text{if}\,\,\lambda'_j=1\\
SJ_{\lambda}(\infty)+SJ_{\lambda^{\sharp}}(\infty),\,\,\text{if}\,\,\lambda'_j>1
\end{cases}
\end{equation}
From the previous formulae it is also easy  to deduce that
\begin{gather}
SJ_{\lambda}(t)=SJ_{\lambda}(\infty)-
\frac{1}{t-l+1}SJ_{\lambda^{\sharp}}(\infty)+\frac{1}{t-l+1}SJ_{\lambda^{2\sharp}}(\infty)\label{infty1}\\
+\dots+(-1)^{l-1} \frac{1}{t-l+1}SJ_{\lambda^{(l-1)\sharp}}(\infty)+(-1)^{l} \frac{2}{t-l+1}SJ_{\lambda^{l\sharp}}(\infty)\notag
\end{gather}

\section{Supercharacters}
Let $\pm\varepsilon,\pm\delta_1,\dots,\pm\delta_n$ be the weights of identical representation of the Lie superalgebra $\frak{osp}(2,2n)$.
The root system of the Lie superalgebra $\frak{osp}(2,2n)$ is 
$$
R_0=\{ \pm\delta_i\pm\delta_j,\,\, i<j,\,\pm2\delta_i\},\,\,\, R_1=\{\pm\varepsilon\pm\delta_i,\,\}
$$
with the bilinear form
$$
(\varepsilon,\varepsilon)=1,\,(\delta_j,\delta_j)=-1, \, (\delta_i,\delta_j)=0,\,i\ne j,\,(\varepsilon,\delta_i)=0,\,1\le i,j\le n
$$
The Weyl group $W_0$  is semi-direct product of $S_n$ and $\Bbb Z_2^n$. It acts on  the wights by permuting and changing the signs of $\delta_j,\,j=1,\dots,n$.
Let us chose the following system of simple roots
$$
B=\{\varepsilon-\delta_1,\delta_1-\delta_2,\dots,\delta_{n-1}-\delta_{n}, 2\delta_n\}, 
$$
We will consider only integer weights
$$
P=\{ \chi=\chi_0\varepsilon+\sum_{j=1}^n\chi_j\delta_j,\chi_0,\chi_j\in \Bbb Z\}
$$
and the subset of the  highest weights
$$
P^+=\{ \chi=\chi_0\varepsilon+\sum_{j=1}^n\chi_j\delta_j,\,\,\chi_1\ge\chi_2\dots\ge \chi_n\ge 0\}
$$

We will denote for any $\chi\in P^+$ by $L(\chi)$  the corresponding finite dimensional irreducible module and by $K(\chi)$ the corresponding Kac module (see (\cite{VDJ})).
We also set
$$
\rho_0=\sum_{i=1}^n(n+1-i)\delta_i,\,\,\rho_1=n\varepsilon,\,\, \rho=\rho_0-\rho_1
$$
$$
L_0=\prod_{\alpha\in R^{+}_0}(e^{\frac12\alpha}-e^{-\frac12\alpha}),\quad L_1=\prod_{\alpha\in R^{+}_1}(e^{\frac12\alpha}-e^{-\frac12\alpha})
$$
A highest weight $\chi$ is called typical if $(\chi+\rho,\alpha)\ne0$ for any $\alpha\in R^+_1$. A highest weight $\chi$ is called atypical if $(\chi+\rho,\alpha)=0$ for some $\alpha\in R^+_1$. In the case of Lie superalgebra $\frak{osp}(2,2n)$ there is at most one such $\alpha\in R^+_1$.  
 We need the following formula for supercharacter of irreducible module $L(\chi)$ by Van Der Jeugt \cite{VDJ}.
 Let us denote by  $\{f\}$   the alternation operation  over $W_0$.
 
 For any $\chi\in P$  we set
$$
 K_{\chi}=e^{\chi+\rho_0}\prod_{\alpha\in R_1^{+}}(1-e^{-\alpha}),
$$
and if $(\chi+\rho,\alpha)=0$ then we set 
$$
 K^{\alpha}_{\chi}=e^{\chi+\rho_0}\prod_{\alpha\in R_1^{+}\setminus\alpha}(1-e^{-\alpha}). 
$$

\begin{lemma}\label{trans}(see \cite{VDJ}) The following equalities hold true:

$1)$  If $\chi\in P^+$ is typical then
\begin{equation}\label{Weyl1}
L_0\,sch\,L(\chi)=\{K_{\chi}\}
\end{equation}
$2)$  If $\chi\in P^+$ is atypical  such that  $(\chi+\rho,\alpha)=0,\,\alpha\in R^{+}_1$ then
\begin{equation}\label{Weyl1}
L_0\,sch\,L(\chi)=\{K_{\chi}^{\alpha}\}
\end{equation}
\end{lemma}

\begin{lemma} The following statements hold true:

$1)$ If $(\chi+\rho,\varepsilon+\delta_j)=(\chi+\rho,\varepsilon+\delta_{j+1})=0$ then
$$
\{K^{\varepsilon+\delta_j}_{\chi}\}=-\{K^{\varepsilon+\delta_{j+1}}_{\chi-\varepsilon-\delta_{j+1}}\}
$$
$2)$ If $(\chi+\rho,\varepsilon+\delta_j)=(\chi+\rho,\varepsilon-\delta_j)=0$ then 
$$
\{K^{\varepsilon+\delta_j}_{\chi}\}=-\{K^{\varepsilon-\delta_{j}}_{\chi-\varepsilon+\delta_{j}}\}
$$
$3)$ 
if $(\chi+\rho,\varepsilon-\delta_j)=(\chi+\rho,\varepsilon-\delta_{j+1})=0$ then
$$
\{K^{\varepsilon-\delta_j}_{\chi}\}=-\{K^{\varepsilon-\delta_{j+1}}_{\chi-\varepsilon+\delta_{j+1}}\}
$$

\end{lemma}
\begin{proof} Let us  prove the first statement. Denote by $A$ the following expression
$$
A=\prod_{\alpha\in R_1^{+}\setminus\{\varepsilon+\delta_j,\,\varepsilon+\delta_{j+1}\}}
$$
Then we have 
$$
\{e^{\chi+\rho_0}(1-e^{-\varepsilon-\delta_{j+1}})A\}+\{e^{\chi-\varepsilon-\delta_{j+1}+\rho_0}(1-e^{-\varepsilon-\delta_{j}})A\}
$$
$$
=\{(e^{\chi+\rho_0}-e^{\chi+\rho_0-2\varepsilon-\delta_{j}-\delta_{j+1}})A\}
$$
And from the conditions of the Lemma it is easy to see that expression  in  brackets is symmetric with respect to transposition $(j,j+1)$. Therefore the result  of alternation is zero. 

Let us prove the second statement. Denote by $B$ the following expression
$$
B=\prod_{\alpha\in R_1^{+}\setminus\{\varepsilon+\delta_j,\,\varepsilon-\delta_{j}\}}
$$
Then we have 
$$
\{e^{\chi+\rho_0}(1-e^{-\varepsilon+\delta_{j}})B\}+\{e^{\chi-\varepsilon+\delta_{j}+\rho_0}(1-e^{-\varepsilon-\delta_{j}})B\}
$$
$$
=\{(e^{\chi+\rho_0}-e^{\chi+\rho_0-2\varepsilon})B\}
$$
And from the conditions of the Lemma it is easy to see that expression  in  brackets is symmetric with respect to transformation  $\delta_j\rightarrow-\delta_j$. Therefore the result  of alternation is zero. 

Let us  prove the third statement. Denote by $A$ the following expression
$$
C=\prod_{\alpha\in R_1^{+}\setminus\{\varepsilon-\delta_j,\,\varepsilon-\delta_{j+1}\}}
$$
Then we have 
$$
\{e^{\chi+\rho_0}(1-e^{-\varepsilon+\delta_{j+1}})C\}+\{e^{\chi-\varepsilon+\delta_{j+1}+\rho_0}(1-e^{-\varepsilon+\delta_{j}})C\}
$$
$$
=\{(e^{\chi+\rho_0}-e^{\chi+\rho_0-2\varepsilon+\delta_{j}+\delta_{j+1}})C\}
$$
And from the conditions of the Lemma it is easy to see that expression  in  brackets is symmetric with respect to transposition $(j,j+1)$. Therefore the result  of alternation is zero. 
Lemma is proved.
\end{proof}

For any diagram $\lambda\in H(1,n)$ let us define the highest weight $\chi_{\lambda}$ by the formula
$$
\chi_{\lambda}=\lambda_1\varepsilon+\sum_{j=1}^n\mu'_j\delta_j
$$
where $\mu$ is the diagram obtaining from $\lambda$ by deleting the first row.
 \begin{proposition}\label{Kac}  Let $\lambda$ be a singular diagram then
$$
sch \,K(\chi_{\lambda})= sch\,L(\chi_{\lambda})+(-1)^{s(\lambda)-s(\lambda^{\sharp})} sch\,L(\chi_{\lambda^{\sharp}})
$$
where for any  for diagram $\lambda\in H(1,n)$    $s(\lambda)=\lambda_2+\lambda_3+\dots$.

\end{proposition}
\begin{proof} Let us consider two cases. First one is when $\lambda'_j>1$ and the second case is when $\lambda'_j=1$.   
Consider now the first case.
Let us temporary denote $\chi_{\lambda}=\chi$. It is easy to check that 
$$
\chi_{\lambda^{\sharp}}=\chi-(\varepsilon+\delta_j)-\dots-(\varepsilon+\delta_{j+r-1})
$$
Therefore
$$
 \{K_{\chi}\}=\{(1-e^{-\varepsilon-\delta_j}) L_{\chi}\}-\{ L_{\chi-\varepsilon-\delta_j}\}
$$
If $\chi-\varepsilon-\delta_j\in P^+$  then $\chi-\varepsilon-\delta_j=\chi_{\lambda^{\sharp}}$ and we proved the proposition. 
If $\chi-\varepsilon-\delta_j\notin P^+$ then 
$$
(\chi-\varepsilon-\delta_j+\rho,\varepsilon+\delta_j)=(\chi-\varepsilon-\delta_j+\rho,\varepsilon+\delta_{j+1})=0
$$
and by previous Lemma we have 
$ 
 \{L_{\chi-\varepsilon-\delta_j}\}=- \{L_{\chi-2\varepsilon-\delta_j-\delta_{j+1}}\}. 
$
And we repeat this procedure until we arrive to $\chi_{\lambda^{\sharp}}$. Besides it is easy to see that $r=s(\lambda)-s(\lambda^{\sharp})$.And we prove the Proposition in the first case.

Now let us consider the second case. It is easy to check that 
$$
\chi_{\lambda^{\sharp}}=\chi-2r\varepsilon,\,\,r=r(\lambda)
$$
As before using the first statement of Lemma \ref{trans} we get 
$$
 \{K_{\chi}\}=\{ L_{\chi}\}+(-1)^r\{ L_{\chi-r\varepsilon-(\delta_j+\dots+\delta_n)}\}
$$
Then using the  second statement  of the same Lemma we get 
$$
\{ L_{\chi-r\varepsilon-(\delta_j+\dots+\delta_n)}\}=-\{ L_{\chi-(r+1)\varepsilon-(\delta_j+\dots+\delta_n)}\}
$$
and by the third  statement  of the same Lemma  we get
$$
\{ L_{\chi-r\varepsilon-(\delta_j+\dots+\delta_n)}\}=(-1)^r\{ L_{\chi-2r\varepsilon}\},\,\,s(\lambda)=s(\lambda^{\sharp})
$$
Proposition is proved.
\end{proof}

\begin{corollary} \label{sing3}If $\lambda$ is singular and $\lambda_1-n=\lambda'_j+n-j$ then the  following equality hold true 
$$
(-1)^{s(\lambda)}sch\,L(\chi_{\lambda})= (-1)^{s(\lambda)}sch\,K(\chi_{\lambda})-(-1)^{s(\lambda^{\sharp})}sch\,K(\chi_{\lambda^{\sharp}})+\dots
$$
$$
+(-1)^{l-1}(-1)^{s(\lambda^{(l-1)\sharp})}sch\,K(\chi_{\lambda^{(l-1)\sharp}})+(-1)^l(-1)^{s(\lambda^{l\sharp})} sch\,L(\chi_{\lambda^{l\sharp}})
$$
where $l=\lambda'_j$.
\end{corollary}

Corollary easily follows from the Proposition \ref{Kac}.

  In order to connect super Jacobi polynomials with representation theory  we need to consider the outer automorphism $\theta$ (see \cite{Serga1}) of the Lie superalgebra $\frak{osp}(2,2n)$ which acts on weights the by $\theta(\varepsilon)=-\varepsilon$ and $\theta(\delta_j)=\delta_j,\,j=1,\dots,n$.
\begin{definition}

Let  $\lambda\in H(1,n)$ then we set
$$
L(\lambda)=\begin{cases} L(\chi_{\lambda}),\,\,\text{if},\,\,\lambda_1\le n\\
L(\chi_{\lambda})\oplus \theta(L(\chi_{\lambda}),\,\,\text{if},\,\,\lambda_1> n
\end{cases}
$$
and 
 $$
E(\lambda)=\begin{cases}L(\chi_{\lambda}),\,\,\text{if}\,\,\lambda_1\le n\\
 K(\chi_{\lambda})+\theta( K(\chi_{\lambda})),\,\, \text{if}\,\,\lambda_1> n
\end{cases}
$$
\end{definition}
We will also consider supercharacters as polynomials in indeterminates
$$
x=e^{\varepsilon},\,x^{-1}=e^{-\varepsilon},\,y_j=e^{\delta_j},\, y_{j}^{-1}=e^{-\delta_j}, \,\, j=1,\dots, n
$$

\begin{thm}\label{Euler} Polynomials $sch\,E(\lambda)$ satisfy the following Pieri identity
$$
sch\,E(\Box) sch\,E(\lambda)=\sum_{\mu\in S(\lambda)}d_{\lambda,\mu} sch\, E(\mu)
$$
where
$$
d_{\lambda,\mu}=\begin{cases}0\,\,\text{if}\,\, \lambda=\mu\,\,\\
0,\,\,\text{if}\,\, \lambda_1=n,\mu_1=n-1,\\
2\,\,\text{if}\,\,\lambda_1=n+1,\mu_1=n\\
(-1)^{s(\lambda)-s(\mu)},\,\,\text{otherwise}
\end{cases}
$$
\end{thm}
\begin{proof} 
We need a formula for  supercharacter of irreducible module   over Lie superalgebra $\frak{osp}(2,2n)$  in case $\lambda_1\le n$.  Let use denote  $u=x+x^{-1}$ and $v_j=y_j+y^{-1}_j,\,j=1,\dots,n$. Then  by  \cite {CL} Proposition 3.1. the following formulae holds true
\begin{equation}\label{LS}
L_0 E(\lambda)=\left\{y_1^{\mu'_1+n}\dots y_n^{\mu'_n+1}\prod_{i=1}^{\lambda_1}(u-v_i)\right\}
\end{equation}
where $L_0=\prod_{i<j}(v_i-v_j)$ and the brackets $\left\{\right\}$ mean the alternation over the Weyl group $S_n\ltimes Z_{2}^n$.

 Let us denote by $B_{\lambda}$ the expression in the  brackets in the formula (\ref{LS}). Then  we have 
$$
L_0sch\,E(\Box) sch\,E(\lambda)=\left\{(u-v_{\lambda_1+1})B_{\lambda}\right\}-\sum_{i\ne \lambda_1+1} \left\{v_iB_{\lambda}\right\}
$$
It is easy to see that
$$
\left\{(u-v_{\lambda_1+1})B_{\lambda}\right\}=\left\{B_{\lambda+\varepsilon}\right\},\,\,\left\{v_jB_{\lambda}\right\}=0,\,\,j>\lambda_1+1,\,\, 
$$
$$
\left\{v_jB_{\lambda}\right\}=\left\{B_{\lambda+\delta_j}\right\}+\left\{B_{\lambda-\delta_j}\right\},\,j<\lambda_1
$$
and for $j= \lambda_1$ and $\mu'_j>0$  we have
$$
\left\{v_jB_{\lambda}\right\}=\left\{B_{\lambda+\delta_j}\right\}+\left\{B_{\lambda-\delta_j}\right\}.
$$
If $\mu'_j=0$ Then we have 
$$
(y_{j}+y_{j}^{-1})y_{j}^{n-j+1}y_{j+1}^{n-j}(u-v_j)=y_{j}^{n-j+2}y_{j+1}^{n-j}(u-v_j)+y_{j}^{n-j}y_{j+1}^{n-j}(u-v_j)
$$
and
$$
y_{j}^{n-j}y_{j+1}^{n-j}(u-v_j)=y_{j}^{n-j}y_{j+1}^{n-j}u-y_{j}^{n-j+1}y_{j+1}^{n-j}-y_{j}^{n-j-1}y_{j+1}^{n-j}
$$
Therefore
$$
\left\{v_jB_{\lambda}\right\}=\left\{v_jB_{\lambda+\delta_j}\right\}+\left\{v_jB_{\lambda-\varepsilon}\right\}
$$

Now let us consider the case $\lambda_1\ge  n$. In this case a proof much easy since the formula for Kac module is more simple.  The condition $d_{\lambda,\lambda}=0$ follows from the fact that characters of  $E(\Box)$ has no constant term.  The condition $d_{\lambda,\mu}=0$ follows from the  fact that  $\theta(sch\,K(\chi_{\lambda}))=sch\,K(\chi_{\mu})$. In the third case condition $d_{\lambda,\mu}=1$ follows from the fact that $\theta(K(\chi_{\mu})=K(\chi_{\mu})$.
\end{proof}

\begin{corollary}\label{last} $SJ_{\lambda}(\infty)=(-1)^{s(\lambda)}sch\,E(\lambda)$
\end{corollary}
\begin{proof}
Let us  take limit  as $p\to 0$ in the formulae  of the Theorem \ref{lim1} then we get 
$$
a_{\lambda,\mu}=\begin{cases}1\,\,\text{if}\,\, \mu\in S^+(\lambda)\\
0\,\,\text{if}\,\, \mu=\lambda\\
1-\delta(\lambda_1-n)+\delta(\lambda_1-n-1)\,\, \text{if}\,\, \mu_1=\lambda_1-1\\
1\,\, \text{otherwise}\,\
\end{cases}
$$
Therefore  $(-1)^{s(\lambda)}sch\,E(\lambda)$ satisfy the same Pieri formulae as $J_{\lambda}(\infty)$ and Corollary follows.
\end{proof} 
\begin{remark}\label{irr} Conditions $\lambda_1>n$ and  $\lambda_1-n\ne \lambda'_j+n-j$ for any $1\le j\le n$ are  equivalent  to the typicality of $\chi_{\lambda}$. Therefore under such conditions $K(\chi_{\lambda})$ is irreducible module over Lie superalgebra $\frak{osp}(2,2n)$ and module $L(\lambda)=K(\chi_{\lambda})\oplus \theta(K(\chi_{\lambda}))$ is irreducible module over Lie supergroup $OSP(2,2n)$. 

We also should note that from  the formula (\ref{projective}) and Corollary  \ref{last} it follows that 
$$
 SI_{\lambda}=(-1)^{s(\lambda)}sch\,E(\lambda)+(-1)^{s(\lambda^{\sharp})}sch\,E(\lambda^{\sharp})
 $$ 
 Therefore from the BGG duality in this case  (see \cite{Z}) it follows that  up to sign the polynomials $SI_{\lambda}$  coincide with  supercharacters of   the projective coves of irreducible finite dimensional modules over supergroup $OSP(2,2n)$. 
\end{remark}
\begin{corollary}\label{irr1} Let $\lambda\in H(1,n)$ such that $\lambda_1-n=\lambda'_j+n-j,\,1\le j\le n$. Then
$$
SJ_{\lambda}(\lambda'_j)=(-1)^{s(\lambda)}sch\,L(\lambda)
$$
\end{corollary}
\begin{proof} The equality follows from Corollary \ref{sing3} and Corollary \ref{last}.
\end{proof}

\section{Acknowledgements}
This work  by A.N. Sergeev was supported  by  Russian Ministry of Education  and  Science  
(grant 1.492.2016/1.4) (sections 2,3,4) and   by the  Russian  Academic Excellence Project '5-100' (sections 5,6).

\end{document}